\documentclass[12pt]{amsart}

\usepackage{amssymb}
\newtheorem{prop}{Proposition}
\newtheorem{lemma}[prop]{Lemma}
\newtheorem{cor}[prop]{Corollary}
\newtheorem{thm}[prop]{Theorem}

\theoremstyle{definition}

\begin{document}
\title[Freiman models in Heisenberg groups]{A note on Freiman models in Heisenberg groups}

\author[N. Hegyv\'ari]{Norbert Hegyv\'ari}
\address{Norbert Hegyv\'{a}ri, ELTE TTK,
E\"otv\"os University, Institute of Mathematics, H-1117
P\'{a}zm\'{a}ny st. 1/c, Budapest, Hungary}
\email{hegyvari@elte.hu}

\author[F. Hennecart]{Fran\c cois Hennecart}
\address{Fran\c cois Hennecart,
PRES Universit\'e de Lyon,
Universit\'e Jean-Monnet,
LAMUSE,
23 rue Michelon,
42023 Saint-\'Etienne, France} \email{francois.hennecart@univ-st-etienne.fr}

\thanks{Research  is partially supported by OTKA grants K 67676, K 81658 and  Balaton Program Project}

\date{\today}

\begin{abstract}
Green and Ruzsa recently proved that for any $s\ge2$, any small squaring set $A$ in a (multiplicative) abelian group, i.e.  $|A\cdot A|<K|A|$, has a Freiman $s$-model: it means that there exists a group $G$ and a Freiman $s$-isomorphism from $A$ into $G$ such that $|G|<f(s,K)|A|$.

In an unpublished note, Green proved that such a result does not necessarily hold in non abelian groups if $s\ge64$.
The aim of this paper is improve Green's result by showing that it remains true under the weaker assumption $s\ge6$.

\end{abstract}

\maketitle

\section{\bf Introduction}

We will use the notation $|X|$ for the cardinality of any set or group $X$.
If $X$ and $Y$ are subsets of a given (multiplicative) group, the product $X\cdot Y$ or simply $XY$
denotes the set $\{xy\,\mid\,x\in X, y\in Y\}$. For $X=Y$ we write
$XY=X^2$. The set $X^{-1}$ is formed by all the inverse elements $x^{-1}$, $x\in X$.

Let
$s\ge2$ be an integer and $A\subset H$ and $B\subset G$ be subsets of  arbitrary (multiplicative) groups. A map $\pi:A\to B$ is said to be a Freiman
$s$-homomorphism
if for any
$2s$-tuple $(a_1,\dots,a_s,b_1,\dots,b_s)$ of elements of $A$ and any
signs $\epsilon_i=\pm1$, $i=1,\dots,s$, we have
$$
a_1^{\epsilon_1}\dots a_s^{\epsilon_s}=b_1^{\epsilon_1}\dots b_s^{\epsilon_s}
\Longrightarrow
\pi(a_1)^{\epsilon_1}\dots \pi(a_s)^{\epsilon_s}=
\pi(b_1)^{\epsilon_1}\dots \pi(b_s)^{\epsilon_s}.
$$
Observe that in the case of abelian groups, we may set, without
loss of generality, all the signs to $+1$.
If moreover $\pi$ is bijective and $\pi^{-1}$ is also a Freiman $s$-homomorphism, then $\pi$ is called a Freiman $s$-isomorphism from $A$
into $G$. In this case, $A$ and B are said to be Freiman $s$-isomorphic.

Green and Ruzsa proved in \cite{GR} that a structural result holds for
small squaring sets in an abelian (multiplicative) group.
The key argument in their proof is Proposition 1.2 of \cite{GR} asserting that any small squaring finite set $A$ in an abelian group has a good Freiman model, that is a relatively small finite group $G$ and a Freiman $s$-isomorphism from $A$ into $G$. More precisely, they showed the following effective result:

{\sl Let $s\ge2$ and $K>1$. There exists a constant $f(s,K)=(10sK)^{10K^2}$ such that $A$ is a subset of an abelian group $H$ satisfying the small squaring property $|A\cdot A|<K|A|$, then
 there exists an abelian group $G$  such that $|G|< f(s,K)|A|$ and
 $A$ is Freiman $s$-isomorphic to a subset of $G$.}

It is not difficult to see that this result cannot be literally extended to nonabelian groups by considering a set $A$  such that
$|A\cdot A|/|A|$ is small and $|A\cdot A\cdot A|/|A|$ is large (see \cite[page 94]{TV}  for such an example).
However it is known (by combining \cite[section 1.11]{Ru} and \cite[Proposition 2.40]{TV}) that if $|A\cdot A|/|A|\le K$ then for any $n$-tuple of signs  $\epsilon_1,\dots,\epsilon_n\in\{-1,1\}$, we have $|X^{\epsilon_1}\cdot X^{\epsilon_2}\cdots X^{\epsilon_n}|/|X|\le K^{O(n)}$ for some large subset $X$ of $A$ satisfying  $|X|\ge|A|/2$.
Despite this fact, the existenceness of a good Freiman $s$-model for some large subset of an arbitrary
set $A_0$ satisfying the small squaring property $|A_0\cdot A_0|<2|A_0|$ is not  guaranteed.
Indeed in his unpublished note \cite{G}, Green gave an example of such a set $A_0$ with arbitrarily large 
cardinality and  the following property:  let $s\ge64$ and $\delta=1/23$;  then for any $A\subset A_0$  
with $|A|\ge |A_0|^{1-\delta}$ and any finite  group $G$  
such that there is a Freiman $s$-isomorphism from $A$ into $G$,
 we have $|G|\ge|A|^{1+\delta}$. 
There is no doubt from his proof that the admissible range for $s$ could be somewhat improved ($s\ge32$ is seemingly the best range that can be read from his proof). 

Our aim is to improve Green's result by showing: 
\begin{thm}\label{th1}
Let $n$  be any positive integer and $\varepsilon$ be any positive real number. Then there exists a finite (nonabelian) group $H$ and a subset $A_0$ in $H$ with the following properties:
\begin{itemize}
\item[i)] $|A_0|>n$ and $|A_0\cdot A_0|<2|A_0|$;
\item[ii)]  For any $A\subset A_0$ with $|A|\ge |A_0|^{43/44}$ and for any finite group $G$ such that there exists a Freiman $6$-isomorphism from $A$ onto $G$, we have $|G|\ge |A|^{33/32-\varepsilon}$.
\end{itemize}
\end{thm}

Our proof in Section \ref{s4} is partially based on Green's approach but also includes new materials. It exploits arguments coming from group theory and Fourier analysis with additional tools, e.g. a recent incidence theorem due to Vinh \cite{LAV}. It also needs some additional combinatorial arguments.

In Section \ref{s3}, we include for comparison the proof of a weaker statement that does not use the new materials, but which optimizes, in some sense, Green's ideas.

Let $p$ be a prime number and $\mathbb{F}$ the fields with $p$ elements.
We denote by $H$ the Heisenberg linear group over $\mathbb{F}$ consisting of the
upper triangular matrices
$$
[x,y,z]=\begin{pmatrix}
1 & x & z\\
0 & 1 & y\\
0 & 0 & 1
\end{pmatrix},\quad x,y,z\in\mathbb{F}.
$$
We recall the product rule in $H$:
$$
[x,y,z]\cdot[x',y',z']=[x+x',y+y',xy'+z+z'].
$$
As shown in \cite{G}, this group provides an example of a nonabelian group in which there exists some subset $A_0$ with small {\it squaring} property, namely $|A_0^2|<2|A_0|$, and not having a good Freiman model. That is there is no {\it relatively big} isomorphic image of $A_0$ by a Freiman $s$-isomorphism with a given $s$ in any group $G$. We will also use the Heisenberg group in order to derive our results.

The proof of Theorem \ref{th1} goes in the following manner. We will show that:
firstly there exists a non trivial $p$-subgroup in the subgroup
generated by $\pi(A)$ in $G$; secondly any
element in $\pi^{-1}(G)$ is the product of at most $6$
elements from $A$ or $A^{-1}$. The rest of the proof is based on some group-theoretical properties which are mainly taken from \cite{G}.

As indicated in \cite{G}, there is no hope to obtain an optimal result by this approach, namely a similar result with $s_0=2$.

\section{\bf Some properties of finite nilpotent groups and of the Heisenberg group $H$}\label{s2}

For any group $G$, we  denote by $1_G$ the identity element of $G$. Thus
$[0,0,0]=1_H$.

We will use the following partially classical properties:
\begin{enumerate}

\item $H$ is a two-step nilpotent group (or nilpotent of class two).
Indeed, the commutator of $a_1=[x_1,y_1,z_1]\in H$ and
$a_2=[x_2,y_2,z_2]\in H$ denoted
by $[a_1;a_2]$ is equal to
$$
[a_1;a_2]=a_1a_2a_1^{-1}a_2^{-1}=[0,0,x_1y_2-x_2y_1].
$$
For any $a_3=[x_3,y_3,z_3]\in H$, we obtain 
$$
[[ a_1;a_2];a_3]=[0,0,0]=1_H,
$$
for the double commutator. Hence the result.

\item Any finite nilpotent group is the direct product of its
Sylow subgroups (see 6.4.14 of \cite{W}).

\item Any finite $p$-group of order $p$ or $p^2$ is abelian
(see 6.3.5 of \cite{W}).

\item Assume that $A\subset H$ and $\pi$ is a Freiman $s$-homomorphism from $A$ into $G$ with $s\ge 5$. We  denote by $\langle\pi(A)\rangle$ the subgroup
generated by $\pi(A)$.
Then $\langle \pi(A)\rangle$ is  a two-step nilpotent group. Indeed, for any $a,b,c\in A$, one has
$$
aba^{-1}b^{-1}c=caba^{-1}b^{-1}
$$
since $H$ is a nilpotent group of class $two$. Hence
$$
\pi(a)\pi(b)\pi(a)^{-1}\pi(b)^{-1}\pi(c)=\pi(c)\pi(a)\pi(b)\pi(a)^{-1}\pi(b)^{-1}
$$
since $\pi$ is a Freiman $s$-homomorphism with $s\ge5$. It thus follows that
double commutators satisfy $[[ a_1;b_1];c_1]=1_G$ for any $a_1,b_1,c_1\in \pi(A)$.
In \cite{G}, the author observed from a direct argument that it remains true for any $a_1,b_1,c_1\in\langle\pi(A)\rangle$: since
$\langle\pi(A)\rangle$ is finite, the result will follow from the next lemma (cf. \cite{G}).

\begin{lemma} \label{lem2}
Let $\Gamma$ be any group and $X$ a maximal subset of $\Gamma$
such that
\begin{equation}\label{new1}
[[ a;b];c]=1_{\Gamma},\quad\text{ for any $a,b,c\in X$.}
\end{equation}
Then $X$ in closed under multiplication.
\end{lemma}

For the the sake of completeness we include the proof which is in the same way
as in \cite{G}.

\begin{proof}
By \eqref{new1} and the following
identity
\begin{equation}\label{new2}
[xy;z]=[x;[y;z]]\cdot[y;z]\cdot[x;z],\quad x,y,z\in\Gamma,
\end{equation}
we obtain
for any $a,b,c,d\in X$, $[[ab;c];d]=[[b;c]\cdot[a;c];d]$. Applying again
\eqref{new2} with $x=[b;c]$, $y=[a;c]$ and $z=c$, yields
in view of \eqref{new1},
\begin{equation}\label{new3}
[[ab;c];d]=1_{\Gamma},\quad \text{for any $a,b,c,d\in X$.}
\end{equation}
By a further application of \eqref{new2} with
$x=a$, $y=b$ and $z=[ab;c]$, we get by \eqref{new3}
$[ab;[ab;c]]=1_{\Gamma}$ for any $a,b,c\in X$. By the maximal property of $X$,
we obtain $ab\in X$ for any $a,b\in X$.
\end{proof}

\end{enumerate}

\section{\bf Approach of the proof with a slightly weaker result}
\label{s3}

Before proving our main result, we explain the principle of the approach by showing the following weaker result in which only Freiman $s$-isomorphisms with $s\ge7$ are considered.

\begin{thm}\label{thm3}
Let $n$ be a positive integer and $\theta$ be a real number such that
$$
\frac{11}{12}\leq\theta\leq1
$$
and let
$$
\varphi_{\theta}=\frac{12\theta-9}{2}.
$$
Then there exists a finite group $H$ and a subset $A_0$ in $H$ satisfying the following properties:
\begin{itemize}
\item[i)] $|A_0|>n$ and $|A_0\cdot A_0|<2|A_0|$;
\item[ii)]  For any $A\subset A_0$ with $|A|\ge |A_0|^{\theta}$ and for any finite group $G$ such that there exists a Freiman $7$-isomorphism from $A$ onto $G$, we have  $|G|\ge |A|^{\varphi_{\theta}}$.
\end{itemize}
\end{thm}

For $\theta=13/14$, it yields the following corollary which can be compared to Theorem \ref{th1}:

\begin{cor}\label{cor4}
Let $n$ be any positive integer. Then  there exists a finite group $H$ and a subset $A_0$ in $H$ satisfying 
the following properties:
\begin{itemize}
\item[i)] $|A_0|>n$ and $|A_0\cdot A_0|<2|A_0|$;
\item[ii)]  For any $A\subset A_0$ with $|A|\ge |A_0|^{13/14}$ and for any finite group $G$ such that there exists a Freiman $7$-isomorphism from $A$ onto $G$, we have  $|G|\ge |A|^{15/14}$.
\end{itemize}
\end{cor}

Let $\alpha\in(0,1)$ and $A_0$ be the subset of $H$
\begin{equation}\label{new4}
A_0:=\{[x,y,z]\,\mid\, (x,y,z)\in [0,p^{\alpha})\times\mathbb{F}\times \mathbb{F}\}.
\end{equation}
For $p$ large enough, we plainly have
$$
|A_0\cdot A_0|=2|A_0|-p^2,
$$
thus $A_0$ is a small squaring subset of $H$.

Let $\theta$ be such that $0<\theta \le1$, on which an additional assumption
will be given later. Let $A$ be any subset of $A_0$ whose cardinality satisfies
\begin{equation}\label{new5}
|A|\ge |A_0|^{\theta}.
\end{equation}

By an averaging argument, there exists $x_0,y_0,z_0,z'_0,u,v\in\mathbb{F}$ and
$X,Y,Z\subset \mathbb{F}$ such that
\begin{align}
&[X,y_0,z_0]\cup[x_0,Y,z'_0]\cup[u,v,Z]\subset A\label{new6}
\\
&|X|\ge \frac{|A|}{p^{2}},\quad
|Y|\ge \frac{|A|}{p^{1+\alpha}},\quad
|Z|\ge \frac{|A|}{p^{1+\alpha}}.\label{new7}
\end{align}
Observe that $|X||Y||Z|^2 \ge p^3$ if
\begin{equation}\label{new8}
|A|\ge p^{(8+3\alpha)/4},
\end{equation}
which holds true if we
fix $\alpha$ such that
\begin{equation}\label{new81}
\theta=\frac{8+3\alpha}{8+4\alpha},
\end{equation}
that is
\begin{equation}\label{new9}
\alpha=\frac{8(1-\theta)}{4\theta-3},
\end{equation}
assuming that the following condition on $\theta$ holds:
$$
\theta\ge\frac{11}{12}.
$$

Let $a=[x,y_0,z_0]$, $b=[x_0,y,z'_0]$. These are elements of $A$.
Moreover the commutator of $a$ and $b$ is
$$
aba^{-1}b^{-1}=[0,0,xy-x_0y_0].
$$
Let $c=[u,v,z]$ and $d=[u,v,z']$ in $[u,v,Z]\subset A$. We thus have
$$
aba^{-1}b^{-1}cd^{-1}=[0,0,xy+z-z'-x_0y_0].
$$

For any element $t$ in $\mathbb{F}$, let $N(t)$ be the number of  representations of $t$ under the form
$$
t=xy+z-z'-x_0y_0,\quad x\in X,\quad y\in Y,\quad z,z'\in Z.
$$
One has 
$$
N(t)=\frac{1}{p}\sum_{h=0}^{p-1}
\sum_{\substack{x\in X\\y\in Y\\z,z'\in Z}}
e\left(\frac{h(xy-x_0y_0+z-z'-t)}{p}\right),
$$
where $e(\alpha)$ is the usual notation for $\exp(2i\pi \alpha)$.
We get
$$
N(t)\ge\frac{|X||Y||Z|^2}{p}-\frac{1}{p}\sum_{h=1}^{p-1}
|S(h)||T(h)|^2,
$$
where
$$
S(h)=\sum_{(x,y)\in X\times Y}e\left(\frac{hxy}{p}\right),
\quad
T(h)=\sum_{z\in Z}e\left(\frac{hz}{p}\right).
$$
By Vinogradov's inequality
$$
|S(h)|\le\sqrt{p|X||Y|}\quad (\text{if }p\nmid h)
$$
and Parseval's identity
$$
\frac{1}{p}\sum_{h=1}^{p}|T(h)|^2=|Z|,
$$
we deduce the lower bound
$$
N(t)>\frac{|X||Y||Z|^2}{p}-\sqrt{p|X||Y|}|Z|.
$$
Hence by \eqref{new9}, $N(t)$ is positive. We thus deduce
$$
[0,0,\mathbb{F}]\subset B:= A^2A^{-2}AA^{-1}.
$$

Let $G$ be any finite group and $\pi$ any Freiman $s$-isomorphism from $A$ into $G$.
Our goal is to show that $|G|$ is big compared to $|A|$. We thus may assume
that $G=\langle \pi(A)\rangle$.

We assume in the sequel that $s\ge7$. We start from the property that is proven just above:
$$
\pi([0,0,\mathbb{F}])\subset \pi(B).
$$

For any $z\in\mathbb{F}$, we let
$$
g_z=\pi([0,0,z]).
$$
If $h=\pi([u,v,w])\in \pi(A)$, then for $s\ge7$ we have
\begin{equation}\label{new10}
\pi([-u,-v,uv-w+z])=\pi([u,v,w]^{-1}[0,0,z])=h^{-1}g_z=g_z h^{-1}.
\end{equation}

We now show that for some $i\ne j$,
$$
g_{\lambda(i-j)}=g_{(\lambda-1)(i-j)}g_{i-j},\quad 0<\lambda\le p.
$$
Since $[u,v,Z]\subset A$ and $|Z|>1$ by \eqref{new7} and \eqref{new8},
$A$ contains at least two distinct elements $[u,v,i]$ and
$[u,v,j]$. We denote $h_k=\pi([u,v,k])$ for $k=i,j$.
Since $\pi$ is a Freiman $s$-isomorphism from $A$
into $G$ and $s\ge7$, we get $h_j^{-1}h_i=g_{i-j}$ and by a similar calculation
as in \eqref{new10}
$$
g_{(\lambda+1)(i-j)}h_i^{-1}=g_{\lambda(i-j)}h_j^{-1},
$$
hence
$$
g_{(\lambda+1)(i-j)}=g_{\lambda(i-j)+j}h_j^{-1}h_i=g_{\lambda(i-j)}g_{i-j}.
$$
We deduce by  induction
$$
g_{\lambda(i-j)}=g_{i-j}^{\lambda},\quad \text{for any $\lambda\ge1$.}
$$
Thus the order of $g_{i-j}$ in $G$ is either $0$ or $p$.
Since $s\ge2$, we have $h_i\ne h_j$ hence $g_{i-j}=h_j^{-1}h_i\ne1_G$. This
shows that $g_{i-j}$ is of order $p$ in $G$. We then deduce that
$p$ divides the order of $G$.

Let $G_p$ be the Sylow $p$-subgroup of $G$. Since $s\ge5$ and $H$ is a two-step nilpotent group, $G$ is also a two-step nilpotent group  by Property 4 of Section \ref{s2}. Then by Property 2 of Section \ref{s2}, $G$ can be written as the direct product $G=G_p\times K$. The projection $\sigma$ of $G$ onto $G_p$ is a homomorphism thus
$\tilde{\pi}=\sigma\circ\pi$ is a Freiman $s$-homomorphism. Since
for $z\ne0$, $h_z$ has order $p$ in $G$, $\sigma(h_z)$ has also order $p$ in $G_p$.

Let $a_1=[x_1,y_1,z_1]$ and $a_2=[x_2,y_2,z_2]$ be any
elements in $A$. We have
$a_1a_2a_1^{-1}a_2^{-1}=[0,0,x_1y_2-x_2y_1]$. If $G_p$ were
abelian  we
would obtain by using $s\ge4$
$$
1_G=\tilde{\pi}(a_1)\tilde{\pi}(a_2)\tilde{\pi}(a_1)^{-1}\tilde{\pi}(a_2)^{-1}=
\tilde{\pi}(a_1a_2a_1^{-1}a_2^{-1})=\tilde{\pi}([0,0,x_1y_2-x_2y_1])
=\sigma(g_{x_1y_2-x_2y_1}),
$$
hence $x_1y_2-x_2y_1=0$. We would conclude that $|A|\le p^2$, a
contradiction by the fact that $|A|\ge|A_0|^{\theta}\ge
p^{(2+\alpha)\theta}>p^2$ by \eqref{new81}.

Consequently by Property 3 given in Section \ref{s2}, $G_p$ is not abelian and $|G_p|\ge p^3$. Finally
$$
|G|\ge p^3=|A_0|^{3/(2+\alpha)}\ge|A|^{(12\theta-9)/2}.
$$
The proof of Theorem \ref{thm3} finishes by choosing the prime $p$ large enough in order to have $|A_0|>n$.

\section{\bf Proof of the main result Theorem \ref{th1}}
\label{s4}

Again, $A_0$ denotes the set
$$
A_0=\{[x,y,z]\,:\, 0\le x< p^{\alpha},\ y,z\in\mathbb{F}\},
$$
and $A$ any subset of $A_0$ such that $|A|\ge|A_0|^{\theta}$.
The parameters $\alpha\in(0,1)$ and $\theta\in(0,1)$ will be specified below.
Again, we have $|A_0|\ge p^{2+\alpha}$ thus
\begin{equation}\label{new11}
|A|\ge p^{(2+\alpha)\theta}.
\end{equation}

We recall that there exist $x_0,y_0,z_0,z'_0,u,v\in\mathbb{F}$
and $X,Y,Z\subset \mathbb{F}$ such that :
\begin{align}
&[X,y_0,z_0]\cup[x_0,Y,z'_0]\cup[u,v,Z]\subset A\nonumber\\
&|X|\ge \frac{|A|}{p^{2}},\quad
|Y|\ge \frac{|A|}{p^{1+\alpha}},\quad
|Z|\ge \frac{|A|}{p^{1+\alpha}}.\label{new12}
\end{align}
For $(x,y,z)\in X\times Y\times Z$, one has
$$
[x,y_0,z_0][x_0,y,z'_0][x,y_0,z_0]^{-1}[x_0,y,z'_0]^{-1}[u,v,z]=[u,v,xy+z-x_0y_0].
$$
Our first goal is to show that $[u,v,t]$ is in $A^2A^{-2}A$ except for $t$ belonging to a small subset $E$ of exceptions.

\noindent \textbf{First step:} For any $t$ in $\mathbb{F}$, let $r(t)$ be the number of triples
$(x,y,z)\in X\times Y\times Z$ such that
$$
t=xy+z-x_0y_0.
$$
One cannot prove that $r(t)>0$ for any $t$. Nevertheless, we will show that except for a small part of elements $t$, this property holds.
Let $C$ be the set of those elements of $t$ for which $r(t)>0$. Then by the Cauchy-Schwarz inequality
\begin{equation}\label{new13}
|C|\ge\frac{(|X||Y||Z|)^2}{\sum_t r(t)^2}.
\end{equation}
Furthermore $\sum_t r(t)^2$ coincides with the number of solutions of
$$
xy+z=x'y'+z',\quad x,x'\in X,\ y,y'\in Y,\ z,z'\in Z.
$$
If we fix $x=x_1$, $x'=x'_1$ and $z'=z'_1$, it gives the equation of an hyperplan $D_{x_1,x'_1,z'_1}$ in $\mathbb{F}^3$ :
$$
x_1y-x'_1y'+z-z'_1=0.
$$
All these hyperplanes are different and there are $|X|^2|Z|$ such hyperplanes.
The possible number of points $(y,y',z)\in Y\times Y\times Z$  is $|Y|^2|Z|$.

In \cite{LAV}, L.A. Vinh established a Szemeredi-Trotter type result
by obtaining an incidence inequality for points and hyperplanes in $\mathbb{F}^{d}$.
It is connected to the Expander Mixing Lemma (see Corollary 9.2.5 in \cite{AS}).
We have:

\begin{lemma}[L.A. Vinh \cite{LAV}]\label{p2}
Let $d\ge2$. Let $\mathcal{P}$ be a set of points in $\mathbb{F}^{d}$ and $\mathcal{H}$ be a set of hyperplanes
in $\mathbb{F}^{d}$. Then
$$
\left|\{(P,D)\in \mathcal{P}\times \mathcal{H}\ :\ P\in D\}\right|
\le \frac{|\mathcal{P}||\mathcal{H}|}{p}+(1+o(1))p^{(d-1)/2}(|\mathcal{P}||\mathcal{H}|)^{1/2}.
$$
\end{lemma}

By this result with $d=3$, we get for any large $p$
$$
\sum_t r(t)^2\le \frac{(|X||Y||Z|)^2}{p}+2p|X||Y||Z|,
$$
which yields by \eqref{new13}
$$
|C|\ge p-\frac{2p^3}{|X||Y||Z|}.
$$
Thus the set $E$ of exceptions $t\in\mathbb{F}$ with $r(t)=0$ has cardinality
\begin{equation}\label{new14}
|E|\le \frac{2p^3}{|X||Y||Z|}.
\end{equation}

\noindent \textbf{Second step:} We fix $z_1$ any element in $Z$ and let
$Z_1=Z\smallsetminus\{z_1\}$. For any $z\in Z_1$, we denote
$$
m(z)=\max\{m\le p \,:\, z_1 +j(z-z_1)\notin E,\ 2\le j\le m\}
$$
if the maximum exists and we let $m(z)=p$ otherwise.
Let
\begin{equation}\label{new15}
T=\left[\frac{|Z_1|}{2|E|}\right]
\end{equation}
If we denote by $Z'_1$ the set of the elements $z\in Z_1$ with $m(z)\le T$, then
$$
|Z'_1|=\sum_{m\le T}|\{z\in Z_1\,:\, m(z)=m\}|\le
\sum_{m\le T}|E|\le \frac{|Z_1|}{2},
$$
since $m=m(z)$ implies $z_1+(m+1)(z-z_1)\in E$. It follows that
$m(z) > T$ for at least one half of the elements $z$ in $Z_1$.
We denote by $\tilde{Z}_1$ the set of those elements $z$. We have
\begin{equation}\label{new16}
|\tilde{Z}_1|\ge \frac{|A|}{2p^{1+\alpha}}.
\end{equation}

\begin{lemma}
Assume that $23/24<\theta\le1$ and let $\gamma$ be a positive real number such that
\begin{equation}\label{new17}
\gamma<\frac{2(2+\alpha)\theta-(3+2\alpha)}{3}.
\end{equation}
If $|E|<p^{\gamma}$, then there exists an integer $t$ with $1\le t\le T$
and two distinct elements $z,z'\in \tilde{Z}_1$ such that
\begin{equation}\label{new18}
z'-z\notin E-E\quad\text{and}\quad z'=z_1+t(z-z_1)
\end{equation}
\end{lemma}

\begin{proof}
For $1\le t\le T$, we denote by $s(t)$ the number of pairs $z,z'$ of elements of $\tilde{Z}_1$ with the required property. It is sufficient to show that
$$
\sum_{t=1}^Ts(t)>0.
$$
This sum can be rewritten as
$$
\sum_{t=1}^T\frac1p\sum_{0\le |h|\le p/2}\sum_{\substack{z,z'\in -z_1+\tilde{Z}_1\\z'-z\notin E-E}}
e\left(\frac{h(z^{-1}z'-t)}{p}\right).
$$
The contribution related to $h=0$ is plainly bigger than
$$
\frac{T}{p}(|\tilde{Z}_1|^2-|\tilde{Z}_1||E-E|),
$$
thus
$$
\sum_{t=1}^Ts(t)\ge \frac{T}{p}(|\tilde{Z}_1|^2-|\tilde{Z}_1||E-E|)
-\frac{1}{p}\sum_{0<|h|< p/2}\Big|\sum_{t=1}^Te\left(\frac{-th}{p}\right)\Big|
\Big|\sum_{\substack{z,z'\in -z_1+\tilde{Z}_1\\z'-z\notin E-E}}
e\left(\frac{hz^{-1}z'}{p}\right)\Big|.
$$
By extending the summation over $z$ and $z'$, we obtain for any  $h\ne0$
$$
\Big|\sum_{\substack{z,z'\in -z_1+\tilde{Z}_1\\z'-z\notin E-E}}
e\left(\frac{hz^{-1}z'}{p}\right)\Big|\le
\Big|\sum_{z,z'\in -z_1+\tilde{Z}_1}e\left(\frac{hz^{-1}z'}{p}\right)\Big|
+|\tilde{Z}_1||E-E|,
$$
which is less than or equals to
$$
(\sqrt{p}+|E-E|)|\tilde{Z}_1|
$$
by using Vinogradov's inequality for the estimation of the sum over $z$ and $z'$. Hence by the bounds
$$
\Big|\sum_{t=1}^Te\left(\frac{-ht}{p}\right)\Big|\le\frac{p}{2|h|}
\quad \text{for} \; 0<|h|< p/2,
$$
and
$$
\sum_{h=1}^{(p-1)/2}\frac{1}{h}\le \ln p,
$$
we get
$$
\sum_{t=1}^Ts(t)\ge \frac{T}{p}(|\tilde{Z}_1|^2-|\tilde{Z}_1||E-E|)
-(\sqrt{p}+|E-E|)|\tilde{Z}_1|\ln p.
$$
From the trivial bound $|E-E|\le |E|^2$ and by \eqref{new15} and \eqref{new16}, this sum is positive whenever
$|E| \le p^{\gamma}$
for $p$ is large enough, where $\gamma$ is any positive number such that
\begin{equation}\label{new19}
\gamma<\min\left(\frac{(2+\alpha)\theta-(1+\alpha)}{2};
\frac{4(2+\alpha)\theta-(7+4\alpha)}{2};
\frac{2(2+\alpha)\theta-(3+2\alpha)}{3}
\right).
\end{equation}
The second argument in this minimum is less than or equal to the first since $\theta\le 1$ and the third
is less than the second since $\theta > 23/24$.
Thus condition \eqref{new19} reduces to \eqref{new17}, and the lemma follows.
\end{proof}

By \eqref{new12} and \eqref{new14}, we deduce from the lemma that the condition
$$
7+2\alpha-3(2+\alpha)\theta< \frac{2(2+\alpha)\theta-(3+2\alpha)}{3},
$$
is sufficient in order to ensure that system
\eqref{new18} has at least one solution, assuming $p$ is large enough. This condition reduces to
$$
\theta>\frac{24+8\alpha}{22+11\alpha}
$$
or equivalently
\begin{equation}\label{new20}
\alpha>\alpha_0(\theta):=\frac{24-22\theta}{11\theta-8}.
\end{equation}
Since $\alpha<1$, we must choose $\theta$ such that
$\theta>\frac{32}{33}.$
Fixing
\begin{equation}\label{new21}
\alpha=\alpha_0(\theta)+\varepsilon,
\end{equation}
this yields
\begin{equation}\label{new22}
p^3\ge|A|^{3/(2+\alpha)}\ge |A|^{3(11\theta-8)/8-\varepsilon},
\end{equation}
for any $p\ge p_0(\epsilon)$.
For $\theta=43/44$, it will give the desired exponents in Theorem \ref{th1}.

\noindent\textbf{Third step:}
We have at our disposal $z_1,z\in Z$ and $t\in\mathbb{F}$ such that
\begin{equation}\label{new23}
z_1+j(z-z_1)\notin E, \quad j=2,\dots,t, \quad\text{and}\quad z_1+t(z-z_1)\in Z.
\end{equation}

Let $\pi:A\to G$, where $G$ is a finite group, be a Freiman $6$-isomorphism. As in the proof of Theorem \ref{thm3},
we will show that $p$ divides $|G|$ and that the $p$-Sylow subgroup of $G$ cannot be abelian. It will ensure the bound $|G|\ge p^3$ and the theorem will follow by \eqref{new22}.

Let
\begin{equation}\label{new24}
h=\pi([0,0,z-z_1])=\pi([u,v,z_1])^{-1}\pi([u,v,z]).
\end{equation}
Let us show that
for any $j$ such that $j(z-z_1)+z_1\notin E$, we have $\pi([0,0,j(z-z_1)])=h^j$.

If $1\le j\le t$, we proceed by induction: for $j=1$, the property is plainly true. Let $2\le j\le t$. We have
$$
\pi([u,v,j(z-z_1)+z_1][u,v,z]^{-1})=\pi([u,v,(j-1)(z-z_1)+z_1][u,v,z_1]^{-1}).
$$
By \eqref{new23} and by definition of $E$, both elements
$[u,v,(j-1)(z-z_1)+z_1]$ and $[u,v,j(z-z_1)+z_1]$ belong to $A^2A^{-2}A$.
Moreover $[u,v,z],[u,v,z_1]\in A$ hence,
by the fact that $\pi$ is a Freiman $6$-homomorphism, we get
$$
\pi([u,v,j(z-z_1)+z_1])\pi([u,v,z])^{-1}=\pi([u,v,(j-1)(z-z_1)+z_1])\pi([u,v,z_1])^{-1}.
$$
Thus, by \eqref{new24}
$$
\pi([u,v,j(z-z_1)+z_1])=\pi([u,v,(j-1)(z-z_1)+z_1])h.
$$
By multiplying on the left by $\pi([u,v,z_1])^{-1}$ and using again that
$\pi$ is a Freiman $6$-homomorphism, we get
$$
\pi([0,0,j(z-z_1)])=\pi([0,0,(j-1)(z-z_1)])h=h^{j}
$$
by the induction hypothesis.

For larger $j$, we again induct: let $j>t$ be such that
$j(z-z_1)+z_1\notin E$. Then at least one of the two elements $(j-1)(z-z_1)+z_1$ or
$(j-t)(z-z_1)+z_1$ is not in $E$ since $z'-z\notin E-E$.

If $(j-1)(z-z_1)+z_1\notin E$ we argue by induction as above. If
$(j-t)(z-z_1)+z_1\notin E$ we slightly modify the argument: since
$$
\pi([u,v,j(z-z_1)+z_1][u,v,t(z-z_1)+z_1]^{-1})=
\pi([u,v,(j-t)(z-z_1)+z_1][u,v,z_1]^{-1})
$$
and $\pi$ a Freiman $6$-isomorphism, we get
\begin{align*}
\pi([u,v,j(z-z_1)+z_1])&=\pi([u,v,(j-t)(z-z_1)+z_1])\pi([u,v,z_1])^{-1}
\pi([u,v,t(z-z_1)+z_1])\\&=
\pi([u,v,(j-t)(z-z_1)+z_1])h^t,
\end{align*}
and finally  by induction
$$
\pi([0,0,j(z-z_1)])=\pi([u,v,z_1])^{-1}\pi([u,v,(j-t)(z-z_1)+z_1])h^t=h^{j-t}h^t=h^j.
$$
Since $z_1\notin E$, we obtain $h^p=1$ in $G$, thus either $h=1$ or $h$ has order $p$. But $z\ne z_1$ hence $[0,0,z-z_1]=[u,v,z][u,v,z_1]^{-1}\ne 1_H$, hence $h\ne 1_G$ since $\pi$ is a Freiman $6$-isomorphism. We deduce that $G$ admits an element of order $p$, thus the $p$-Sylow subgroup $G_p$ of $G$ is not trivial. By considering the
canonical homomorphism $\sigma:G\to G_p$, $\tilde{\pi}=\sigma\circ\pi$ is
a Freiman $6$-homomorphim of $A$ onto $G_p$. Hence for any $a=[x,y,z]$ and
$b=[x',y',z']$ in $A$
$$
[ \tilde{\pi}(a);\tilde{\pi}(b)]= \tilde{\pi}([ a;b])=\tilde{\pi}([0,0,xy'-x'y])
$$
which must be equal to $1_G$ if $G_p$ is assumed to be abelian. It would mean
that $(x,y)$ belongs to a single line for any  $[x,y,z]\in A$, giving
$|A|\le p^2$ a contradiction to
$$
\frac{\ln|A|}{\ln p}\ge \theta(2+\alpha)>\theta(2+\alpha_0(\theta))=\frac{8\theta}{11\theta-8}>2,
$$
obtained by \eqref{new11}, \eqref{new20} and \eqref{new21}.

\end{document}